\documentclass[12pt,epsfig,amsfonts]{amsart} 
\usepackage{amsmath,amsthm,amssymb,amscd,epsfig}
\usepackage{graphicx}

\newtheorem*{theorema}{Theorem}

\newtheorem{lemma}{Lemma}[section]

\newtheorem{prop}[lemma]{Proposition}
\newtheorem{definition}[lemma]{Definition}

\begin{document}
\author{hiroki takahasi}

\address{Department of Mathematics,
Keio University, Yokohama,
223-8522, JAPAN} 
\email{hiroki@math.keio.ac.jp}
\subjclass[2010]{37D25, 37E30, 37G25}

\title[Birkhoff spectrum for H\'enon-like maps] 
{Birkhoff spectrum for H\'enon-like maps \\
at the first bifurcation}

\begin{abstract}
We effect a multifractal analysis for a strongly dissipative H\'enon-like map at the first bifurcation parameter
at which the uniform hyperbolicity is destroyed by the formation of tangencies inside the limit set.
We decompose the set of non wandering points on the unstable manifold 
  into level sets of Birkhoff averages of continuous functions, and
  derive a formula for
the Hausdorff dimension of the level sets in terms of the
entropy and unstable Lyapunov exponent of invariant probability measures.
  \end{abstract}

\maketitle

\section{introduction}
The multifractal analysis of chaotic dynamical systems consists in the study of fine geometric structures of invariant sets. 
One considers the so-called multifractal decompositions
of an invariant set, and the associated multifractal spectra which encodes this decomposition.
By connecting the spectra to other characteristics of the system,
such as entropy and Lyapunov exponents of invariant measures,
one tries to get more refined description of the underlying dynamics than purely stochastic considerations.

In this paper we treat the Birkhoff averages of continuous functions.
Although this type of problem is well-understood for uniformly hyperbolic systems,
 much less is known for non hyperbolic ones.
We treat certain non-hyperbolic two-dimensional maps at the boundary of uniform hyperbolicity,
having quadratic tangencies
between invariant manifolds.

We are concerned with a family of H\'enon-like diffeomorphisms
\begin{equation*}\label{henon}
f_a\colon(x,y)\in\mathbb R^2\mapsto(1-ax^2,0)+b\cdot\Phi(a,b,x,y),\quad a\in\mathbb R, \ 0<b\ll1.\end{equation*}
Here, $\Phi$ is bounded continuous in $(a,b,x,y)$ and 
$C^2$ in $(a,x,y)$. 
We assume there exists a constant $C>0$ such that for all $a$ near $2$
and small $b$,
\begin{equation*}\label{jacobian}\|D\log|\det Df_{a}|\|\leq C.\end{equation*}
This family describes the transition from uniformly hyperbolic to non hyperbolic regimes.
It is known \cite{BedSmi06,CLR08,DevNit79,Tak13} that there is
a {\it first bifurcation parameter} $a^*=a^*(b)$ with the following properties:
the non wandering
set of $f_a$ is a uniformly hyperbolic horseshoe for $a>a^*$;
for $a=a^*$ there is a unique orbit of homoclinic or heteroclinic tangency,
and the tangency is quadratic.
The aim of this paper is to perform the multifractal analysis of $f_{a^*}$.
Although the dynamics of $f_{a^*}$ resembles that of the horseshoe before the first bifurcation, 
the presence of tangency presents novel obstructions
for understanding the global dynamics.

We state our settings and goals in more precise terms.
Write $f$ for $f_{a^*}$. 
Let $P$, $Q$ denote the fixed saddles of $f$ near $(1/2,0)$, $(-1,0)$ respectively.
The orbit of tangency intersects a small neighborhood of the origin exactly at one point, denoted by $\zeta_0$ (FIGURE 1).
If $f$ preserves orientation, then $\zeta_0\in W^s(Q)\cap W^u(Q)$.
If $f$ reverses orientation, then $\zeta_0\in W^s(Q)\cap W^u(P)$.
The sole obstruction to uniform hyperbolicity is the orbit of the tangency $\zeta_0$.

Let $\Omega$ denote
the non wandering set of $f$, which is a compact set. 
If $f$ preserves orientation, let $W^u=W^u(Q)$. Otherwise,
let $W^u=W^u(P)$.
The (non-uniform) expansion along $W^u$ is responsible for the chaotic behavior of $f$.
Hence, a good deal of multifractal information of $\Omega$ 
is contained in its unstable slice
$$\Omega^u=\Omega\cap W^u.$$

\begin{figure}
\begin{center}
\includegraphics[height=6.5cm,width=14cm]
{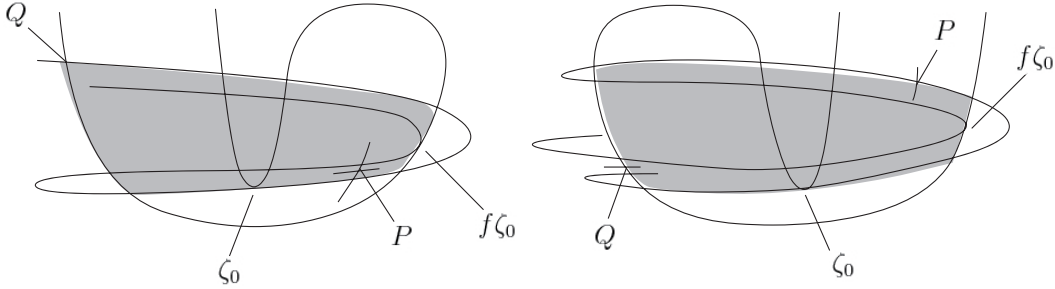}
\caption{Organization of invariant manifolds of the 
fixed saddles $P$, $Q$ of $f=f_{a^*}$ near $(1/2,0)$, $(-1,0)$
respectively. In the case $\det Df>0$ (left),
the stable and unstable manifolds of $Q$ meet each other tangentially. In the case $\det Df<0$
(right), the stable manifold of $Q$ meets the unstable manifold of $P$ tangentially. The shaded
regions represent the region $R$ containing the non wandering set $\Omega$ (see Sect.\ref{family}).}
\end{center}
\end{figure}

Given a continuous function $\varphi\colon \Omega\to \mathbb R$ consider {\it level sets} of the Birkhoff averages of $\varphi$:
$$\Omega_\varphi^u(\beta)=\left\{x\in \Omega^u\colon\lim_{n\to\infty}
\frac{1}{n}S_n\varphi(x)=\beta\right\},\ \ \beta\in\mathbb R,$$
where $S_n\varphi=\sum_{i=0}^{n-1}\varphi\circ f^i$.
Define
$$c_\varphi=\inf_{x\in \Omega}\liminf_{n\to\infty}\frac{1}{n}S_n\varphi(x)
\ \text{ and } \ d_\varphi=\sup_{x\in\Omega}\limsup_{n\to\infty}\frac{1}{n}S_n\varphi(x).$$
In what follows we assume 
 $c_\varphi<d_\varphi$. 
 For otherwise the Birkhoff averages of $\varphi$ along
all orbits are equal. 
Set $I_\varphi=[c_\varphi,d_\varphi]$. Consider the multifractal decomposition 
\begin{equation*}\label{decomposition}\Omega^u
=\left(\bigcup_{\beta\in I_\varphi}
\Omega_\varphi^u(\beta)\right)\cup \hat \Omega_\varphi^u,\end{equation*} 
where
$\hat \Omega_\varphi^u$ denotes the set of points in $\Omega^u$ for which 
$(1/n)S_n\varphi(x)$ does not converge.
This decomposition has extremely complicated topological structures:
each $\Omega_\varphi^u(\beta)$ is nonempty (See Sect.\ref{lowest} for details);
one can show that each set appearing in the decomposition is dense in $\Omega^u$;
namely, a decomposition into an uncountable number of dense subsets.

Let 
$\mathcal M(f)$ denote the set of $f$-invariant
Borel probability measures. The entropy of $\mu\in\mathcal M(f)$ is denoted by $h(\mu)$.
An {\it unstable Lyapunov exponent} of $\mu\in\mathcal M(f)$ is the number $\lambda^u(\mu)$
defined by 
$$\lambda^u(\mu)=\int\log J^u(x) d\mu(x).$$
Here, $J^u(x)=\Vert D_xf|E^u_x\Vert $, and
$E^u_x$ is a one-dimensional subspace of $T_x\mathbb R^2$ called
 an \emph{unstable direction} at $x\in\Omega$ that is characterized by
the following backward contraction property  \cite{SenTak1}:
\begin{equation*}
\limsup_{n\to\infty}\frac{1}{n}\log\|D_xf^{-n}|E^u_x\|<0.\end{equation*}
By a result of \cite{CLR08}, $\inf\{\lambda^u(\mu)\colon\mu\in\mathcal M(f)\}>0$. 
Relationships between entropy, unstable Lyapunov exponents and dimension of
invariant probability measures were established in \cite{LedYou85}.
Our main result connects these characteristics to the Hausdorff dimension of $\Omega_\varphi^u(\beta)$
that is defined as follows.
 Given $p\in(0,1]$
 the unstable Hausdorff 
$p$-measure of a set $A\subset W^u$ is defined by
$$m_p^u(A)=\lim_{\delta\to0}\left(\inf\sum_{U\in\mathcal
U } {\rm length}(U)^p \right),$$ where ${\rm length}(\cdot)$ denotes the length on $W^u$ with respect to the induced Riemannian metric,
and the infimum is taken over all
coverings $\mathcal U$ of $A$ by open sets of $W^u$ with
length $\leq\delta$. The unstable Hausdorff dimension of $A$,
denoted by $\dim_H^u$, is the unique number in $[0,1]$ such that
$$\dim_H^u(A)=\sup\{p\colon
m_p^u(A)=\infty\}=\inf\{p\colon m_p^u(A)=0\}.$$

Now, set
$$B_\varphi^u(\beta)=\dim_H^u(\Omega_\varphi^u(\beta)),$$
and
$$I_\varphi'=\{\beta\in I_\varphi
\colon B_\varphi^u(\beta)\leq 2/\log(1/ b)\}.$$
Our main result is stated as follows.

\begin{theorema}
Let $b>0$ be sufficiently small and $f=f_{a^*(b)}$ as above. 
For any continuous function $\varphi\colon \Omega\to\mathbb R$ and
all $\beta\in I_\varphi\setminus I_\varphi'$,
$$B_\varphi^u(\beta)=\lim_{\varepsilon\to0}\sup\left\{\frac{h(\mu)}{\lambda^u(\mu)}\colon \mu\in\mathcal M(f),\left|\int\varphi d\mu-\beta\right|<\varepsilon\right\}.$$
\end{theorema}

This type of formula has been proved under different settings and 
 assumptions on the hyperbolicity of the systems:
  uniformly hyperbolic ones
 \cite{BarSau01,Ols03, PesWei01, Wei99} (a more complete list of previous results
  can be found in \cite{Pes97});
 maps with parabolic fixed points \cite{Chu10,JohJorObePol10};
certain non-uniformly expanding quadratic maps on the interval \cite{Chu10,ChuTak13}. 
Up to present,
 many of the known results for non hyperbolic systems are limited to the Lyapunov spectrum 
  \cite{GelPrzRam10,GelRam09,IomTod11,KesStr04,Tak13'}.
  
  In \cite{Chu10}, Chung established the formula for a class of one-dimensional maps
  admitting ``nice" induced Markov maps.   
 A strategy for a proof of our theorem is to
 use a (locally defined) stable foliation to identify points on the same leaf
(called {\it long stable leaves} in \cite[Sect.2.8]{Tak13'}), and to extend the one-dimensional argument
 in \cite{Chu10}.
 The same strategy has been taken in
  \cite{Tak13'} in which the non continuous function $\log J^u$ was treated instead of $\varphi$.
  Since the stable foliation is not globally defined,
 it is not possible to tell whether such a leaf through a given point exist.
 The argument in \cite{Tak13'} to handle this difficulty consists of three steps:
 (i)  introduce {\it dynamically critical points} in the spirit of Benedicks and Carleson \cite{BenCar91},
 and define a {\it bad set} in terms of the recurrence to the critical points;
 (ii) show that long stable leaves exist for points outside of the bad set;
 (iii) show that the Birkhoff averages of $\log J^u$ do not converge on the bad set. 
  A novel obstruction in dealing with continuous $\varphi$ is that the Birkhoff averages of $\varphi$ can converge,
  for points in the bad set (denoted by 
  $\Omega_*^u$ in Sect.\ref{outline}). 
 What we can do at best is to show that the dimension of $\Omega_*^u$ is small, and
 establish the formula for those $\beta$ for which $B_\varphi(\beta)$ is not too small.
 This is the reason for the restriction on $\beta$ in the theorem.

To clarify the range of $\beta$ for which the formula in the theorem holds,
  let us recall the thermodynamic formalism of $f$
 developed in \cite{SenTak1,SenTak2}.
 For $t\in\mathbb R$ define
 $$P(t)=\sup\left\{h(\mu)-t\lambda^u(\mu)\colon\mu\in\mathcal M(f)\right\}.$$
 A measure which attains this supremum is called an {\it equilibrium measure} for $-t\log J^u$.
The function $t\mapsto P(t)$ is convex. One has $P(0)>0$, and Ruelle's inequality \cite{Rue78} gives $P(1)\leq0$.
Since $f$ has no SRB measure \cite{Tak12}, $P(1)<0$ holds. Hence the equation $P(t)=0$ has a unique solution in $(0,1)$,
denoted by $t^u$.
 There exists a unique equilibrium measure for $-t^u\log J^u$
 (\cite[Theorem A]{SenTak2}),
 denoted by $\mu_{t^u}$, and $t^u=\dim_H^u(\Omega^u)$, $t^u\to1$ as $b\to0$ (\cite[Theorem B]{SenTak2}).
 From the theorem and the Ergodic Theorem, 
$B_\varphi^u(\int\varphi d\mu_{t^u})\geq h(\mu_{t^u})/\lambda(\mu_{t^u})=t^u$.
It follows that $B_\varphi^u$ takes its maximum at $\beta=\int\varphi d\mu_{t^u}$.
Similarly to the proof of \cite[Theorem C]{Tak13'}
one can show that $B_\varphi^u$ is continuous on $I_\varphi\setminus I_\varphi'$,
increasing on $\{\beta\in I_\varphi\setminus I_\varphi'\colon
\beta<\int\varphi d\mu_{t^u}\}$ and decreasing on  $\{\beta\in I_\varphi
\setminus I_\varphi'\colon
\beta>\int\varphi d\mu_{t^u}\}$, so that
the set $I_\varphi\setminus I_\varphi'$ is an interval containing $t^u$.

 The rest of this paper consists of two sections.
Sect.2 is a preliminary, and the theorem is proved 
in Sect.3.

\section{Preliminaries}
The main reference of this section is \cite{Tak13'}. We collect several results and constructions, and  
prove two lemmas needed for the proof of the theorem.

Throughout this paper we shall be concerned with positive constants $\delta$, $b$, 
chosen in this order.
 The letter $C$ is used to denote any positive constant which is independent of 
 $\delta$ or $b$.

\subsection{The non wandering set}\label{family}

By a {\it rectangle} we mean any
compact domain bordered by two compact curves in $W^u$ and two in the
stable manifolds of $P$ or $Q$. By an {\it unstable side} of a
rectangle we mean any of the two boundary curves in $W^u$. A {\it
stable side} is defined similarly.

By the result of \cite[Lemma 3.2]{Tak13} there exists a rectangle $R$ contained in the set 
$\{(x,y)\in \mathbb R^2\colon |x|<2, |y|< \sqrt{b}\}$ with the following properties (See FIGURE 1):

\begin{itemize}
\item $\displaystyle{\Omega=\{x\in R\colon f^nx\in R\ \text{ for every }n\in\mathbb Z\}}$;

\item one of the unstable sides of $R$ contains $\zeta_0$;

\item one of the stable sides of $R$ contains $f\zeta_0$. This side is denoted by $\alpha_0^+$. The other side, denoted by $\alpha_0^-$,  contains $Q$;

\item $f\alpha_0^+\subset\alpha_0^-$.
\end{itemize}

\subsection{Critical points}\label{critical}

Set
$$I(\delta)=\{(x,y)\in R\colon |x|<\delta\}.$$
The derivatives grow exponentially, as long as the orbit is outside of $I(\delta)$.
To treat returns to $I(\delta)$ we mimic the strategy of Benedicks $\&$ Carleson \cite{BenCar91} and
 introduce the notion of critical points. 
The reference for the contents in this subsection is \cite[Sect.2.4 $\&$ Sect.2.5]{Tak13'}.

From
the hyperbolicity of the saddle $Q$,
there exist two mutually disjoint connected open sets $U^-$, $U^+$ independent of $b$ such that
$\alpha_0^-\subset U^-$, $\alpha_0^+\subset U^+$, $U^+\cap fU^+=\emptyset=U^+\cap fU^-$ and 
a foliation $\mathcal F^s$ of $U=U^-\cup U^+$ by one-dimensional leaves such
that: 
\begin{itemize}
\item $\mathcal F^s(Q)$, the leaf of $\mathcal F^s$ containing $Q$,
contains $\alpha_0^-$; 
\item if $x,fx\in U$, then $f(\mathcal F^s(x))
\subset\mathcal F^s(fx)$;

\item Let $e^s(x)$ denote the unit vector in $T_x\mathcal F^s(x)$ whose second component is positive. 
Then $x\mapsto e^s(x)$ is $C^{1}$, $\|D_xfe^s(x)\|\leq Cb$ and $\|D_xe^s(x)\|\leq C$;

\item If $x,fx\in U$, then $s(e^s(x))\geq
C/\sqrt{b}.$
\end{itemize}
\begin{definition}
{\rm We say $\zeta\in W^u\cap I(\delta)$ is a {\it critical point} if $f\zeta\in U^+$ and
$T_{f\zeta}W^u=T_{f\zeta}\mathcal F^s(f\zeta)$.}
\end{definition}

From the first two conditions on $\mathcal F^s$ and $f\alpha_0^+\subset\alpha_0^-$, there is a leaf of $\mathcal F^s$ which 
contains $\alpha_0^+$. Since $f\zeta_0\in\alpha_0^+$ we have 
$f\zeta_0\in U^+$ and
$T_{f\zeta_0}W^u=T_{f\zeta_0}\mathcal F^s(f\zeta_0)$,
namely, $\zeta_0$ is a critical point.

To locate all other critical points we need some preliminary considerations.
Let $\alpha_1^+$ denote the (connected) component of $W^s(P)\cap R$ containing $P$,
and $\alpha_1^-$ the component of $f^{-1}\alpha_1^+\cap R$ not containing $P$. 
Let $\Theta$ denote the rectangle bordered by $\alpha_1^-$, $\alpha_1^+$ and the unstable sides of $R$.
Let $\tilde\Gamma^u$ denote the collection of 
components of $\Theta\cap W^u$. 
By a \emph{$C^2(b)$-curve} we mean a compact, nearly horizontal $C^2$ curve in $R$ such that the slopes of its tangent directions are $\leq\sqrt{b}$ and the
curvature is everywhere $\leq\sqrt{b}$. 
Let $S$ denote the compact lenticular domain 
which is bounded by the parabola $f^{-1}\alpha_0^+\cap R$ and the unstable side of $R$ containing $\zeta_0$.
Then the following holds \cite[Lemma 2.5 $\&$ Lemma 2.8]{Tak13'}:

\begin{itemize}
\item {\bf (Location)} any element of $\tilde\Gamma^u$ 
is a $C^2(b)$-curve with endpoints in $\alpha_1^-$, $\alpha_1^+$,
and contains a unique critical point;
\item {\bf (Non recurrence)}
all critical points are contained in $S$. 
\end{itemize}
In particular, all critical points never return to the interior of $R$ under forward iteration.
The dynamics of $f$ is amenable to analysis primarily due to this non-recurrence of critical points.
To recover the loss of derivatives suffered from the return to $I(\delta)$, we bind the point
to a suitable critical point \cite[Lemma 2.9]{Tak13'}, and let it copy the exponential derivative growth along the critical orbit
\cite[Lemma 2.6]{Tak13'}.

\subsection{Inducing}\label{induced map}
We introduce an inducing scheme associated with the first return map to $\Theta$.
The reference for the contents in this subsection is \cite[Sect.2.10]{Tak13'}.

Define a sequence $\{\tilde\alpha_n\}_{n=0}^\infty$ of compact curves in $R\cap W^s(P)$ inductively as follows.
First, set $\tilde\alpha_0=\alpha_1^+$. Given $\tilde\alpha_{n-1}$, define $\tilde\alpha_n$ to be one of the two connected components of $f^{-1}\tilde\alpha_{n-1}\cap R$ which is at the left of $\zeta_0$. Observe that $\tilde\alpha_1=\alpha_1^-$. By the Inclination Lemma, 
the Hausdorff distance between $\tilde\alpha_n$ and $\alpha_0^-$ converges to $0$ as $n\to\infty$.

For each $n\geq0$ let $\alpha_n$ denote the connected component of $R\cap f^{-1}\tilde\alpha_n$ 
which is not $\tilde\alpha_{n+1}$. 
The set $R\cap f^{-1}\alpha_n$ consists of two curves, one at the left
of $\zeta_0$ and the other at the right. They are denoted by $\alpha_{n+1}^-$,
$\alpha_{n+1}^+$ respectively.
By definition, these curves obey the following diagram
\begin{equation*}
\{\alpha_{n+1}^-,\alpha_{n+1}^+\}\stackrel{f^2}{\to}\tilde\alpha_n
\stackrel{f}{\to}\tilde\alpha_{n-1}
\stackrel{f}{\to}\tilde\alpha_{n-2}\stackrel{f}{\to}\cdots
\stackrel{f}{\to}\tilde\alpha_1=\alpha_1^-\stackrel{f}{\to}
\tilde\alpha_0=\alpha_1^+.\end{equation*}

Define $r\colon \Theta\to \mathbb N\cup\{\infty\}$ by
$$r(x)=\inf(\{n>0\colon f^nx\in\Theta\}\cup\{\infty\}),$$
which is the first return time of $x$ to $\Theta$.
Note that:
\begin{itemize}

\item $r(x)=1$ if and only if $x\in\alpha_1^-\cup\alpha_1^+$;
$r(x)=n+1$ $(n\geq1)$ if and only if $x$ is sandwiched by $\alpha_{n}^+$ and $\alpha_{n+1}^+$, or
by $\alpha_{n}^-$ and $\alpha_{n+1}^-$;
 $r(x)=\infty$ if and only if $x\in S$;

\item each level set of $r$ except $S$ has exactly two connected components.

\end{itemize}

Let $\mathcal P$ denote the partition of the set $\Theta\setminus(S\cup\alpha_1^-\cup\alpha_1^+)$ into connected components of the level sets of 
the function $r$. 
The $\mathcal P$ is well-defined because the Hausdorff distance between
$\alpha_n$ and $\alpha_0^+$ converges to $0$ as $n\to\infty$. 
 Set $\mathcal P_1=\{\omega=\overline{\eta}\colon\eta\in\mathcal P\}$,
where the bar denotes the closure operation.
For each $n\geq2$ define
$$\mathcal P_n=\left\{\omega_0\cap \bigcap_{i=1}^{n-1} f^{-r(\omega_0)}\circ f^{-r(\omega_1)}
\circ\cdots\circ f^{-r(\omega_{i-1})}\omega_i
\colon\omega_0,\omega_1,\ldots,\omega_{n-1}\in \mathcal P_1\right\}.$$
Elements of $\bigcup_{n\geq0}\mathcal P_n$ are called \emph{proper rectangles.}
The unstable sides of a proper rectangle are formed by two curves contained in the unstable sides of $\Theta$.
Its stable sides are formed by two curves contained in $W^s(P)$.

On the interior of each $\omega\in\mathcal P_1$, the value of $r$ is constant. This value is denoted by $r(\omega)$.
For each $\omega\in\mathcal P_n$ 
define its \emph{inducing time} $\tau(\omega)$ by
\begin{equation*}\label{tau}\tau(\omega)=\sum_{i=0}^{n-1}r(\omega_i).\end{equation*}
Clearly, the unstable sides of $f^{\tau(\omega)}\omega$ are formed by two curves in $\tilde\Gamma^u$.
Its stable sides are formed by two curves contained in the stable sides of $\Theta$ (See FIGURE 2).

The next bounded distortion result is contained in \cite[Lemma 2.15]{Tak13'}.

\begin{lemma}\label{global}
For any $\gamma^u\in\tilde\Gamma^u$ and any proper rectangle $\omega$,
$\gamma^u\cap\omega$ is a compact curve joining the stable sides of $\omega$. In addition,
$$\sup_{x,y\in\gamma^u\cap\omega}
\frac{\|D_yf^{\tau(\omega)}|E_y^u\|}{\|D_xf^{\tau(\omega)}|E_x^u\|}        \leq
C|f^{\tau(\omega)}x-f^{\tau(\omega)}y|.$$
\end{lemma}

\begin{figure}
\begin{center}
\includegraphics[height=4cm,width=10cm]{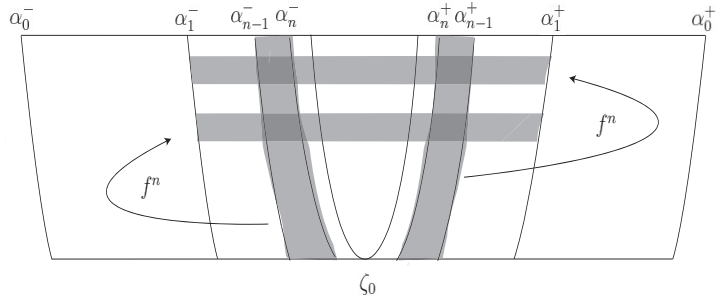}
\caption{The proper rectangles (shaded) in $\mathcal P_1$ with inducing time $n$ and their $f^n$-images}
\end{center}
\end{figure}

\subsection{Horseshoes}\label{horse}
We introduce a horseshoe structure which naturally comes from the inducing scheme in 
Sect.\ref{induced map}. The reference for the contents in this subsection is \cite[Sect.2.10]{Tak13'}.

Let $\mathcal A$ be a finite collection of
proper rectangles contained in the interior of $\Theta$, 
labeled with $1,2,\ldots,\ell=\#\mathcal A$. We assume any two elements of $\mathcal A$ are either disjoint,
or intersect each other only at their stable sides.
Endow $\Sigma_{\ell}=\{1,\ldots,\ell\}^{\mathbb Z}$ with the product topology 
of the discrete topology,
and let $\sigma\colon\Sigma_{\ell}\circlearrowleft$ denote the left shift.
Define a coding map $\pi\colon\Sigma_{\ell}\to\mathbb R^2$ by
$\pi(\{x_i\}_{i\in\mathbb Z})=y$, where
$$\{y\}=\left(\bigcap_{k=1}^{\infty}\omega_k^s\right)
\cap\left(\bigcap_{k=1}^{\infty}\omega_k^u\right)$$ 
and
 $$\omega^s_k=\omega_{x_0}\cap\left(\bigcap_{i=1}^k 
f^{-\tau(\omega_{x_0})}\circ\cdots\circ f^{-\tau(\omega_{x_{i-1}})}\omega_{x_i}\right)\text{ and }
\omega^u_k=\bigcap_{i=1}^{k}f^{\tau(\omega_{x_{-1}})}\circ\cdots\circ f^{\tau(\omega_{x_{-i}})}\omega_{x_{-i}}.$$

\begin{lemma}\cite[Lemma 2.19]{Tak13'}\label{hyp}
The map $\pi$ is well-defined, continuous, injective, and satisfies $\pi(\Sigma_{\ell})\subset\Omega$.
\end{lemma}

\subsection{Bounded distortions}\label{bdddist}
We need two more distortion results for points which are slow recurrent to the critical set.
For $x\in\Omega^u$ define 
$$d_{\rm crit}^u(x)=\begin{cases}|\zeta(x)-x|&\text{ if $x\in I(\delta)$};\\
1&\text{ otherwise,}\end{cases}$$
where $\zeta(x)$ is the critical point
on the $C^2(b)$-curve in $\tilde\Gamma^u$ containing $x$.
The function $d_{\rm crit}^u$ is a ``distance to the critical set".
For each $m\geq0$ define 
$$G_m^u=
\{x\in\Omega^u\colon \text{$d_{\rm crit}^u(f^{n}x)>  b^{\frac{n}{9}}$ for every $n\geq m$}\}.$$
The next lemma, the proof of which is a slight modification of \cite[Lemma 2.20]{Tak13'} and hence omitted here,
gives a distortion bound for derivatives along the unstable direction. 
\begin{lemma}\label{lyapbdd}
For every $m\geq0$ there exists a constant $D_m>0$ such that for
any proper rectangle $\omega$ intersecting $G_m^u$ and $\tau(\omega)>m$,
$$\sup_{x,y\in\Omega\cap\omega}\frac{\|D_yf^{\tau(\omega)}|E_y^u\|}{   \|D_xf^{\tau(\omega)}|E_x^u\|  } \leq D_m.$$
\end{lemma}

The next lemma gives a distortion bound for Birkhoff averages of H\"older continuous functions.



\begin{lemma}\label{variation}
If $\varphi\colon\Omega\to\mathbb R$ is H\"older continuous, then 
for every $m\geq0$ there exists 
$K_{m,\varphi}>0$ such that
for
any proper rectangle $\omega$ intersecting $G_m^u$ and $\tau(\omega)>m$, 
$$\sup_{x,y\in\Omega\cap\omega}|S_{\tau(\omega)}\varphi(x)-S_{\tau(\omega)}\varphi(y)|\leq K_{m,\varphi}.$$
\end{lemma}

\begin{proof}
Let $x\in\Omega\cap\omega$. There exists a nearly horizontal $C^1$-curve, denoted by $\gamma^u(f^{\tau(\omega)}x)$ (called a {\it long unstable leaf through $f^{\tau(\omega)}x$} in \cite{Tak13'}),
which is contained in $f^{\tau(\omega)}\omega$, joins the stable sides of $\Theta$ and 
satisfies
${\rm length}(f^{-n}\gamma)\leq C\rho_0^n$ for some $C>0$, $\rho_0\in(0,1)$ and all $n\geq0$.
Let $\hat x$ denote the point of intersection between $\gamma^u(f^{\tau(\omega)}x)$ and $\alpha_1^+$,
and set $x'=f^{-\tau(\omega)}\hat x$. 
Let $y\in\Omega\cap\omega$ and define $y'$ in the same way.
Let $\theta\in(0,1]$ be a H\"older exponent of $\varphi$.
We have
\begin{equation*}
 \sum_{n=0}^{\tau(\omega)-1}|f^nx-f^ny|^\theta\leq
 \sum_{n=0}^{\tau(\omega)-1}|f^nx-f^nx'|^\theta+\sum_{n=0}^{\tau(\omega)-1}|f^nx'-f^ny'|^\theta+
 \sum_{n=0}^{\tau(\omega)-1}|f^ny'-f^ny|^\theta.\end{equation*}
From the backward contraction, the first and the third summands are uniformly bounded.
For the second one,
by \cite[Lemma 2.18]{Tak13'} there exists $k\in[0,m]$ such that
the curve $f^{k-\tau(\omega)}\alpha_1^+$ is contracted exponentially by a factor $b^{\frac{1}{2}}$
under forward iteration (i.e., $f^{k-\tau(\omega)}\alpha_1^+$ is contained in a {\it long stable leaf} \cite[Sect.2.8]{Tak13'}). 
 Hence
 \begin{equation*}
 \sum_{n=0}^{\tau(\omega)-1}|f^nx-f^ny|^\theta\leq
 2C^\theta\sum_{n=0}^{\infty}\rho_0^{\theta n}+\sum_{n=0}^{k-1}|f^nx'-f^ny'|^\theta+
 C^\theta \sum_{n=0}^{\infty}b^{\frac{\theta n}{2}},\end{equation*}
which is bounded by
 a uniform constant depending only on $b$, $m$, $\theta$.
 Since $\omega$ and $x$, $y$ are arbitrary, the desired inequality follows.
\end{proof}

\noindent{\it Remark.
{\rm The function $x\in\Omega\mapsto \log \|D_xf|E_x^u\|$ is not covered by
Lemma \ref{variation} since it is not continuous at $Q$.}}

\subsection{Approximation of non ergodic measures by ergodic ones}\label{nonergod}
Let $\mathcal M^e(f)$ denote the set of $f$-invariant ergodic Borel probability measures.
\begin{lemma}\label{approximate}
For any continuous $\varphi\colon\Omega\to\mathbb R$, $\mu\in\mathcal M(f)$ and $\varepsilon>0$
there exists $\nu\in\mathcal M^e(f)$ such that
 $h(\nu)>0$, $|h(\mu)- h(\nu)|<\varepsilon$, $|\lambda^u(\mu)-\lambda^u(\nu)|<\varepsilon$
 and $|\int\varphi d\mu-\int\varphi d\nu|<\varepsilon$.
 \end{lemma}
 
 \begin{proof}
 From \cite[Lemma 2.23]{Tak13'} there exists $\nu\in\mathcal M(f)$ such that
 $h(\nu)>0$, $\nu\{Q\}=0$,
  $|h(\mu)-h(\nu)|<\varepsilon$ and $|\lambda^u(\mu)-\lambda^u(\nu)|<\varepsilon.$ 
We note that 
$f|\Omega$ is a factor of the full shift on two symbols
 \cite[Proposition 3.1]{SenTak2}, and therefore has the specification 
 \cite[Proposition 1(b)]{Sig74}.
Hence, ergodic measures are entropy-dense
\cite{EizKifWei94}:
there exists a sequence
$\{\xi_n\}_n$ in $\mathcal M^e(f)$ such that $\xi_n\to\nu$ and $h(\xi_n)\to h(\nu)$ as $n\to\infty$. 
By \cite[Lemma 4.4]{SenTak1} and $\nu\{Q\}=0$, we obtain  $\lambda^u(\xi_n)\to\lambda^u(\nu)$. 
 \end{proof}

 \section{Proof of the theorem}
In this section we complete the proof of the theorem.

\subsection{Outline}\label{outline}
For the rest of this paper we assume $\varphi\colon \Omega\to\mathbb R$ is continuous
 and $\beta\in I_\varphi\setminus I_\varphi'$.
 We prove the theorem by estimating $B_\varphi^u(\beta)$ from both sides.
  This is done along the line of \cite{Tak13'},
  but there is one key difference.

In Sect.\ref{lowest} we estimate $B_\varphi^u(\beta)$ from below,
 by constructing a large subset of $\Omega_\varphi^u(\beta)$.
  For the upper estimate, define
$$\Omega_*^u=\Omega^u\setminus\bigcup_{m=0}^\infty G_m^u,$$
and split $\Omega_\varphi^u(\beta)=\Pi_1\cup\Pi_2,$
where $$\Pi_1=\bigcup_{m=0}^\infty \Omega^u(\beta)\cap G_m^u\ \text{ and }\ \Pi_2=\Omega^u(\beta)\cap\Omega_*^u.$$
The upper estimate of $\dim_H^u(\Pi_1)$ is done in Sect.\ref{up1} 
in much the same way as in \cite{Tak13'}.
The key difference from \cite{Tak13'} is that $\Pi_2$ can be nonempty. 
To bypass this problem, in Sect.\ref{dimens} we estimate the dimension of the larger set $\Omega_*^u$ 
from above.

\subsection{Lower estimate of $B_\varphi^u(\beta)$}\label{lowest}
   Define
\begin{equation}\label{Feps}
d^{u}_\varepsilon={\sup}\left\{\frac{h(\mu)}{\lambda^u(\mu)}\colon\mu\in\mathcal
M(f),\ \left|\int\varphi d\mu
-\beta\right|<\varepsilon\right\}.\end{equation}
We also define $d^{u,e}_\varepsilon$ by restricting the range of the supremum to $\mathcal M^e(f)$.
We shall show 
\begin{equation}\label{lowe}
B_\varphi^u(\beta)\geq \displaystyle{\lim_{\varepsilon\to0}d^{u,e}_\varepsilon}.\end{equation}
Since
$\displaystyle{\lim_{\varepsilon\to0}d^u_\varepsilon=
\lim_{\varepsilon\to0}d^{u,e}_\varepsilon}$ from Lemma \ref{approximate}, 
the desired lower estimate of $B_\varphi^u(\beta)$ follows.
The idea is to construct a sequence of horseshoes in the sense of Sect.\ref{horse} with Birkhoff averages arbitrarily close to $\beta$,
and then glue these horseshoes together to construct a set of points whose Birkhoff averages are precisely $\beta$.

Let  $\{\mu_n\}_n$ be a sequence in $\mathcal M^e(f)$ such that
$|\int\varphi d\mu_n-\beta|\to0$ 
and  $h(\mu_n)/\lambda^u(\mu_n)$ converges as $n\to\infty$.
Since $\varphi$ is continuous and $\mathcal M(f)$ is compact with respect to the topology of weak convergence,
$c_\varphi=\min\{\int\varphi d\mu\colon\mu\in\mathcal M(f)\}$ and 
$d_\varphi=\max\{\int\varphi d\mu\colon\mu\in\mathcal M(f)\}$.
Since
both $c_\varphi$ and $d_\varphi$ are attained by elements of $\mathcal M^e(f)$,
considering linear combinations of them and then 
using Lemma \ref{approximate} one can show there indeed exists such a sequence.
If $h(\mu_n)\to0$, then $h(\mu_n)/\lambda^u(\mu_n)\to0$ from
$\inf\{\lambda^u(\mu)\colon\mu\in\mathcal M(f)\}>0$, and so \eqref{lowe} is obvious.
Hence we assume $h(\mu_n)>0$.
In what follows we first assume $\varphi$ is H\"older continuous, and prove \eqref{lowe}. 
Lastly we indicate necessary minor modifications to treat merely continuous $\varphi$.

If $\varphi$ is H\"older continuous, then 
slightly modifying the proof of \cite[Lemma 2.21]{Tak13'} one can prove a variant of well-known Katok's theorem \cite[Theorem S.5.9]{KatHas95}:
for each $n$ there exist a positive integer $q_n$ and a family $\mathcal R_n$
of proper rectangles with the following properties:

\begin{itemize}
\item[(i)] for each $\omega\in\mathcal R_n$,
$\tau(\omega)=q_n$;

\item[(ii)] $|(1/q_n)\log \#\mathcal R_n-h(\mu_n)|<1/n;$

\item[(iii)] for any $x\in \bigcup_{\omega\in\mathcal R_n} \Omega\cap\omega$, 
$\left|(1/q_n)S_{q_n}\log J^u(x)-\lambda^u(\mu_n)\right|<1/n$;

\item[(iv)] for any $x\in \bigcup_{\omega\in\mathcal R_n} \Omega\cap\omega$, 
$\left|(1/q_n)S_{q_n}\varphi(x)-\int \varphi d\mu_n\right|<1/n$.

\end{itemize}
The only one difference from \cite[Lemma 2.21]{Tak13'} is (iv), which follows from Lemma \ref{variation}.

 The rest of the proof proceeds much in parallel to that of \cite{Tak13'}, and so we only give a sketch of the proof.
 For an integer $\kappa\geq1$ let
$$\mathcal R_n(\kappa)=\{\omega_0\cap f^{-q_n}\omega_1\cap\cdots\cap f^{-(\kappa-1)q_n}\omega_{\kappa-1}
\colon\omega_1,\ldots,\omega_{\kappa-1}\in\mathcal R_n\}.$$
Let $\{\kappa_n\}_{n=1}^\infty$ be a sequence of positive integers.
For each $k\geq1$ let $(N,s)$ be a pair of integers such that
$
k=\kappa_1+\kappa_2+\cdots+\kappa_{N-1}+s \ \text{and}\ \ 0\leq s< \kappa_{N}.$
Define $\mathcal S(k)$ to be the collection of all proper rectangles of the form
$$\omega_0\cap f^{-\kappa_1q_1}\omega_1\cap
\cdots\cap f^{-\kappa_1q_1-\cdots-\kappa_{N-1}q_{N-1}}\omega_{N},$$
where $\omega_n\in\mathcal R_{n}(\kappa_{n+1})$ $(n=0,\ldots,N-1)$ and 
$\omega_{N}\in\mathcal R_{N}(s)$.
The set $\bigcup_{\omega\in\mathcal S(k)}\omega$ is compact, and
decreasing in $k$.
Set
$$Z=\gamma^u(\zeta_0)\cap \bigcap_{k=1}^\infty\bigcup_{\omega\in\mathcal S(k)}\omega,$$
where $\gamma^u(\zeta_0)$ denotes the unstable side of $\Theta$ containing $\zeta_0$.
By appropriately choosing $\{\kappa_n\}_{n=1}^\infty$ so that the orbits of points in $Z$ spend longer and longer times
around the horseshoes as $n$ increases,
one can make sure that
$Z\subset \Omega_\varphi^u(\beta)$ and
\begin{equation*}
\dim_H^u(Z)\geq \lim_{n\to\infty}\frac{h(\mu_n)}{\lambda^u(\mu_n)}.\end{equation*}
Since $\{\mu_n\}_n$ is arbitrary, \eqref{lowe} holds.

If $\varphi$ is merely continuous, then take a sequence $\{\varphi_n\}_n$ of real-valued H\"older continuous 
functions on $\Omega$
 such that $\sup\{|\varphi(x)-\varphi_n(x)|\colon x\in\Omega\}<1/n$.
 Find $q_n$ and $\mathcal R_n$ as above, satisfying (i) (ii) (iii) and (iv)
 with $\varphi_n$ in the place of $\varphi$.
 Then
\begin{align*}\left|\frac{1}{q_n}S_{q_n}\varphi(x)-\int \varphi d\mu_n\right|<&\left|\frac{1}{q_n}(S_{q_n}\varphi(x)-  S_{q_n}\varphi_n(x)) \right|\\
& +\left|\frac{1}{q_n}S_{q_n}\varphi_n(x)-\int \varphi_n d\mu_n\right|+\left|\int\varphi_n -\varphi d\mu_n\right|\\
 <&\frac{1}{n}+\frac{1}{n}+\frac{1}{n},\end{align*}
and so the same argument prevails. \qed

\subsection{Upper estimate of $\dim_H^u(\Pi_1)$.}\label{up1}
From the next proposition and the countable stability of $\dim_H^u$,
we obtain $\dim_H^u(\Pi_1)\leq \displaystyle{\lim_{\varepsilon\to0}d^u_\varepsilon}$.

\begin{prop}\label{c}
 For every  $m\geq0$,
$\dim_H^u(\Omega^u(\beta)\cap G_m^u)\leq \displaystyle{\lim_{\varepsilon\to0}d^u_\varepsilon}.$
\end{prop}

\begin{proof}
Recall that $\gamma^u(\zeta_0)$ denotes the unstable side of $\Theta$ containing $\zeta_0$.
Since $\gamma^u(\zeta_0)$ contains a fundamental domain in $W^u$,
for any $x\in \Omega_\varphi^u(\beta)$ which is not the fixed point in $W^u$ there exists $n\in\mathbb Z$ such that
$f^nx\in\gamma^u(\zeta_0)$. 
From the countable stability and the $f$-invariance
of $\dim_H^u$,
$B_\varphi^u(\beta)=
\dim_H ^u(\Omega_\varphi^u(\beta)\cap\gamma^u(\zeta_0))$.

Set
$$\tilde\Omega_\varphi^u(\beta)=\{x\in\Omega_\varphi^u(\beta)\cap \gamma^u(\zeta_0)\colon f^nx\in\Theta\text{ for infinitely many $n>0$}\}.$$
Since points in $\Omega_\varphi^u(\beta)\cap\gamma^u(\zeta_0)$ which return to $\Theta$ under forward iteration
only finitely many times form a countable subset,
we have $B_\varphi^u(\beta)=\dim_H^u (\tilde\Omega_\varphi^u(\beta))$.
From now on we restrict ourselves to $\tilde\Omega_\varphi^u(\beta)$.

 For $c>0$ let $D_c(\zeta_0)$ denote the closed ball in $\gamma^u(\zeta_0)$ of
  radius $c$ about $\zeta_0$. Define
$$\mathcal A_{n,\varepsilon}=\left\{
 \omega\in\mathcal P_{n}\colon \omega\cap G_m^u\neq\emptyset,\ \omega\cap
 D_c(\zeta_0)=\emptyset,\
\inf_{x\in\omega\cap\gamma^u(\zeta_0)}\left|\frac{1}{\tau(\omega)}S_{\tau(\omega)}\varphi(x)-\beta\right|<\frac{\varepsilon}{2}\right\}.$$
Observe that $\mathcal A_{n,\varepsilon}$ is a finite set, because its elements 
do not intersect $D_c(\zeta_0)$.
For each $\omega\in\mathcal A_{n,\varepsilon}$
write $\omega^u=\omega\cap\gamma^u(\zeta_0)$ and set
$\mathcal A^u_{n,\varepsilon}=\{\omega^u\colon\omega\in\mathcal A_{n,\varepsilon}\}$.
Clearly we have
$$(\tilde\Omega^u(\beta)\cap G_m^u)\setminus D_c(\zeta_0)\subset\limsup_{n\to\infty}\bigcup_{
\omega^u\in\mathcal A_{n,\varepsilon}^u}\omega^u,$$
and there exist $C>0$ and $\rho_1\in(0,1)$ such that
for each $\omega\in\mathcal A_{n,\varepsilon}^u$,
$${\rm length}(\omega^u)\leq C\rho_1^n.$$
It is enough to show
\begin{equation}\label{pperd}
\limsup_{n\to\infty}\frac{1}{n}\log\sum_{\omega^u\in\mathcal A^u_{n,\varepsilon}}{\rm length}(\omega^u)^{ d^u_\varepsilon}\leq 0\ \text{ for any $\varepsilon>0$.}
\end{equation}
Indeed, if this holds, then 
for any $d>0$ we have
$$\limsup_{n\to\infty}\frac{1}{n}\log\sum_{A\in\mathcal A^u_{n,\varepsilon}}{\rm length}(\omega^u)^{ d^u_\varepsilon+d}\leq-d\log\rho_1.$$
It follows 
that $\sum_{A\in\mathcal A^u_{n,\varepsilon}}{\rm length}(\omega^u)^{ d^u_\varepsilon+d}$
has a negative growth rate as $n$ increases.
Therefore
the Hausdorff $( d^u_\varepsilon+d)$-measure of the set $(\tilde\Omega^u(\beta)\cap G_m^u)\setminus
D_c(\zeta_0)$
is $0$. Since $d>0$ is arbitrary,
$\dim_H^u((\tilde\Omega^u(\beta)\cap G_m^u)\setminus D_c(\zeta_0))\leq d^u_\varepsilon$, and 
by the countable stability of $\dim_H^u$ we obtain
$\dim_H^u(\tilde\Omega^u(\beta)\cap G_m^u)\leq d^u_\varepsilon$.
 Letting
$\varepsilon\to0$ yields the desired inequality in Proposition \ref{c}. 
\medskip

It is left to prove \eqref{pperd}.
Set $\ell=\#\mathcal A_{n,\varepsilon}$ and
Write
$\mathcal A_{n,\varepsilon}=\{\omega{(1)},\omega{(2)},\ldots,\omega{(\ell)}\}$
so that 
\begin{equation}\label{align}
\tau(\omega{(1)})\geq\tau(\omega{(s)})>m\ \text{ for every }s\in \{1,2,\ldots,t\}.\end{equation}
Let $\pi_{\ell}\colon\Sigma_\ell\to\bigcup_{\omega\in\mathcal A_{n,\varepsilon}}\omega$ denote the coding map defined in Sect.\ref{horse}
and  $\sigma\colon\Sigma_\ell\circlearrowleft$ the left shift.
Define 
$$B=\{\underline{a}\in\Sigma_\ell\colon\pi\underline{a}\subset W^s(P)\setminus\{P\}\}.$$
Proper rectangles can intersect each other only at their stable sides, and there is only one proper rectangle containing $P$
in its stable side. Hence,
for any $\underline{a}\in \Sigma_\ell\setminus B$ there exists a unique element 
of $\mathcal A_{n,\varepsilon}$ containing $\pi\underline{a}$ which we denote by
$\omega(\underline{a})$.
Define 
$\Psi\colon \Sigma_\ell\setminus B\to\mathbb R$ by
$$\Psi(\underline{a})=- d^u_\varepsilon\sum_{i=0}^{\tau(\omega(\underline{a}))-1}\log J^u(f^i(\pi\underline{a})).$$
Since $\pi(\Sigma_\ell)\subset\Omega\setminus\{Q\}$ and $\log J^u$ is continuous except at $Q$,
$\Psi$ is continuous.

Let
 $\mathcal M(\sigma)$ denote the space of $\sigma$-invariant Borel probability measures on $\Sigma_\ell$
 endowed with the topology of weak convergence.
For each $k\geq 1$ define an atomic probability measure $\nu_k\in\mathcal M(\sigma)$ concentrated on 
the set $E_k=\{\underline{a}\in \Sigma_\ell\colon \sigma^{k}\underline{a}=\underline{a}\}$ by
$$\nu_k=  
\left(\sum_{\underline{b}\in E_k}\exp\left(     S_{k}\Psi(\underline{b})            \right)\right)^{-1}\sum_{\underline{a}\in E_k}
\exp\left(S_{k}\Psi(\underline{a})\right)
\delta_{\underline{a}},$$
where $S_{k}\Psi=\sum_{i=0}^{k-1}\Psi\circ\sigma^i$ and
$\delta_{\underline{a}}$ denotes the Dirac measure at $\underline{a}$.
Let $\nu_0$ denote an accumulation point of the 
sequence $\{\nu_k\}_k$ in $\mathcal M(\sigma)$. Taking
a subsequence if necessary we may assume $\nu_k\to\nu_0$.
We have $\nu_0\in\mathcal M(\sigma)$.
Define a Borel probability measure $\overline{\mu}$ on $\pi(\Sigma_\ell)$ by 
$$\overline{\mu}=\sum_{\omega\in\mathcal A_{n,\varepsilon}}\nu_0|_{\pi^{-1}\omega}.$$
By \cite[Sublemma 3.5]{Tak13'},
$\nu_0(B)=0$ and so
 $\overline{\mu}$ is indeed a probability.
Define $\mu\in\mathcal M(f)$ by
 $$\mu=\left(\sum_{\omega\in\mathcal A_{n,\varepsilon}}
 \tau(\omega)\overline{\mu}(\omega)\right)^{-1}\sum_{\omega\in\mathcal A_{n,\varepsilon}}\sum_{i=0}^{\tau(\omega)-1}(f^i)_*(\overline{\mu}|_{\omega}).$$ 

We show
\begin{equation}\label{minus}
h(\mu)- d^u_\varepsilon\lambda^u(\mu)\leq0.
\end{equation}
To show this,
let $\omega\in\mathcal A_{n,\varepsilon}$ and
$x\in \omega$. Choose
$y\in \omega\cap\gamma^u(\zeta_0)$ such that 
$\left|(1/\tau(\omega))S_{\tau(\omega)}\varphi(y)-\beta\right|<\varepsilon/2.$
If $\varphi$ is H\"older continuous, then by Lemma \ref{lyapbdd}
and  $\tau(\omega)\geq 2n$
we have
\begin{align*}\left|\frac{1}{\tau(\omega)}S_{\tau(\omega)}\varphi(x)-\beta\right|\leq&
\left| \frac{1}{\tau(\omega)}(S_{\tau(\omega)}\varphi(x)-S_{\tau(\omega)}\varphi(y))\right|+
 \left|\frac{1}{\tau(\omega)}S_{\tau(\omega)}\varphi(y)-\beta\right|\\
 \leq& \frac{K_{m,\varphi}}{\tau(\omega)}+\frac{\varepsilon}{2}\leq \frac{K_{m,\varphi}}{2n}+\frac{\varepsilon}{2}   <\varepsilon.\end{align*}
  If $\varphi$ is merely continuous, then approximating $\varphi$ by a H\"older continuous 
 function we get the same inequality for sufficiently large $n$.
Since $\omega\in\mathcal A_{n,\varepsilon}$ and $x\in\omega$ are arbitrary, this implies
 $ |\int \varphi d\mu-\beta|<\varepsilon$.
Then \eqref{minus} follows from the definition of $d^u_\varepsilon$ in \eqref{Feps}.

Observe that 
\begin{align*}
\log\sum_{\underline{a}\in E_k}\exp(S_{k}\Psi(\underline{a}))
=-\sum_{\underline{a}\in E_k}\nu_k(\{\underline{a}\})
\log\nu_k(\{\underline{a}\})+k\int\Psi d\nu_k.\end{align*}
A slight modification of the argument in \cite[pp.220]{Wal82} shows that for
any integer $p$ with $1\leq p<k$,
\begin{equation}\label{x}
\frac{1}{k}\log\sum_{\underline{a}\in E_k}\exp(S_{k}\Psi(\underline{a}))\leq-\frac{1}{p}\sum_{\underline{a}\in E_p}\nu_k(\{\underline{a}\})
\log\nu_k(\{\underline{a}\})+\int\Psi d\nu_k
+\frac{2p\log\#E_p}{k}.\end{equation} 
Similarly to the proof of \cite[Sublemma 3.7]{Tak13'} one can show that
$\int\Psi d\nu_k\to\int \Psi d\nu_0$ as $k\to\infty$.
Letting $k\to\infty$ in \eqref{x},
\begin{equation*}
\limsup_{k\to\infty}\frac{1}{k}\log\sum_{\underline{a}\in E_k}\exp(S_{k}\Psi(\underline{a}))\leq
-\frac{1}{p}\sum_{\underline{a}\in E_p}\nu_0(\{\underline{a}\})
\log\nu_0(\{\underline{a}\})+\int\Psi d\nu_0.
\end{equation*}
Letting $p\to\infty$ we get
\begin{equation}\label{lem2}
\limsup_{k\to\infty}\frac{1}{k}\log\sum_{\underline{a}\in E_k}\exp(S_{k}\Psi(\underline{a}))\leq
h(\sigma;\nu_0)+\int\Psi d\nu_0,
\end{equation}
where $h(\sigma;\nu_0)$ denote the entropy of $\nu_0\in\mathcal M(\sigma)$.

To estimate the left-hand-side of \eqref{lem2} from below, set $E_k'=\{\underline{a}\in E_k\colon a_0=1\}.$
Let 
$\underline{a}\in E_k'$, $\underline{b}\in E_{k-1}'$ be such that
$a_i=b_i$ for every $0\leq i< k-1$.
Slightly modifying the proof of \cite[Sublemma 3.8]{Tak13'} one can show that
$\pi\underline{a}$ and $\pi\underline{b}$ are contained in the same proper rectangle
 with inducing time $>m$ and intersecting $G_m^u$. 
 Lemma \ref{lyapbdd} gives 
 $$\frac{\exp(S_{k}\Psi(\underline{a}))}{   \exp(S_{k-1}\Psi(\underline{b})) }=\frac{\exp(S_{k-1}\Psi(\underline{a}))}{   \exp(S_{k-1}\Psi(\underline{b})) }\cdot\exp(S_{0}\Psi(\sigma^{k-1}\underline{a}))\geq D_m^{-d^u_\varepsilon}\cdot D_m^{-2d^u_\varepsilon}{\rm length}(\omega^u(a_{k-1}))^{d^u_\varepsilon}.$$
Using this inequality repeatedly gives
\begin{align*}
\sum_{\underline{a}\in E_k}\exp(S_k\Psi(\underline{a}))&>
\sum_{\underline{a}\in E_k'} \exp(S_{k}\Psi(\underline{a}))=
\sum_{\underline{b}\in E_{k-1}'} \exp(S_{k-1}\Psi(\underline{b}))\sum_{\stackrel{\underline{a}\in E_k'}
{a_i=b_i\ 0\leq \forall i< k-1}}\frac{ \exp(S_{k}\Psi(\underline{a}))}
{ \exp(S_{k-1}\Psi(\underline{b}))}\\
&\geq\sum_{\underline{b}\in E_{k-1}'} \exp(S_{k-1}\Psi(\underline{b}))\cdot D_m^{-3d^u_\varepsilon}\sum_{\omega\in\mathcal A_{n,\varepsilon}}
{\rm length}(\omega^u)^{ d^u_\varepsilon}\\
&\geq\cdots\geq
\sum_{\underline{b}\in E_{1}'} \exp(S_{0}\Psi(\underline{b}))\left(D_m^{-3d^u_\varepsilon}\sum_{\omega\in\mathcal A_{n,\varepsilon}}
{\rm length}(\omega^u)^{d^u_\varepsilon}\right)^{k-1}\\
&\geq  \left(D_m^{-3d^u_\varepsilon}\sum_{\omega\in\mathcal A_{n,\varepsilon}}
{\rm length}(\omega^u)^{d^u_\varepsilon}\right)^{k}.\end{align*}
Hence
\begin{equation}\label{lem1}\liminf_{k\to\infty}\frac{1}{k}
\log\sum_{\underline{a}\in E_k}\exp(S_k\Psi(\underline{a}))\geq\log\sum_{\omega^u\in\mathcal A^u_{n,\varepsilon}}
{\rm length}(\omega^u)^{d^u_\varepsilon}- 3d^u_\varepsilon\log D_m.\end{equation}
Putting \eqref{lem2}  \eqref{lem1} together and then using \eqref{minus} yield
\begin{align*}
\frac{1}{n}\log\sum_{\omega\in\mathcal A^u_{n,\varepsilon}}
{\rm length}(\omega^u)^{d^u_\varepsilon}&\leq
\frac{1}{n}\left(h(\sigma;\nu_0)+\int \Psi d\nu_0\right)+\frac{3}{n} d^u_\varepsilon\log D_m\\
&=\frac{1}{n}(h(\mu)- d^u_\varepsilon\lambda^u(\mu))\sum_{\omega\in\mathcal A_{n,\varepsilon}}\tau(\omega)\overline{\mu}(\omega)+ \frac{3}{n} d^u_\varepsilon\log D_m\\
&   \leq \frac{3}{n} d^u_\varepsilon\log D_m.\end{align*}
This implies \eqref{pperd}, and hence finishes the proof of Proposition \ref{c}.
\end{proof}

\subsection{Upper estimate of $\dim_H^u(\Omega_*^u)$}\label{dimens}
We finish by proving the next
\begin{prop}\label{new}
$\dim_H^u(\Omega_*^u)\leq 2/\log (1/b)$.
\end{prop}

\begin{proof} 
If $x\in\Omega_*^u$, then there exist infinitely many $n\geq0$ such that
$d_{\rm crit}(f^nx)\leq b^{\frac{n}{9}}.$
Define a sequence $k_i=k_i(x)$ $(i=1,2,\ldots)$ of positive integers inductively as follows: 
$k_1=\min\{n>0\colon d_{\rm crit}(f^nx)\leq b^{\frac{n}{9}}\}.$
Given $k_1,\ldots,k_{i}$ with $d_{\rm crit}(f^{k_1+\cdots+k_i}x)\leq  b^{\frac{k_i}{9}}$,
define $k_{i+1}=\min\{n>0\colon d_{\rm crit}(f^{k_1+\cdots+k_i+n}x)\leq b^{\frac{n}{9}}\}.$

Define $\eta=\eta(b)\gg1$ by
\begin{equation}
\eta=\left[-\frac{1}{20}\log b\right],
\end{equation}
where $[$ $\cdot$ $]$ denotes the integer part.
Since $b^{\frac{k_i}{9}}\cdot \|Df^{2\eta k_i}\|<  b^{\frac{k_i}{9}}\cdot 5^{2\eta k_i} \ll1$, 
$f^{k_1+\cdots+k_i}x$ shadows the forward orbit of the binding critical point 
at least up to time $2\eta k_i$, namely
\begin{equation}\label{d1}
k_{i+1}(x)\geq 2\eta k_i(x).\end{equation}
From \eqref{d1} and $k_1(x)>1$
we get $k_{i}(x)\geq(2\eta)^{i-1}$, and
\begin{equation}\label{s2}k_{1}(x)+k_{2}(x)+\cdots+k_{i}(x)\geq
\eta^{i-1}.\end{equation}

Now, given a sequence $\{l_i\}_{i=1}^\infty$ of positive integers, define a collection 
$\mathcal Q(l_1,l_2,\ldots,l_i)$ 
of pairwise disjoint compact curves in $\gamma^u(\zeta_0)$ inductively as follows. Start with
$$\mathcal Q(l_1)= \{\gamma_{1}\subset\gamma^u(\zeta_0)\colon f^{l_1}\gamma_1\in\tilde\Gamma^u\text{ and $k_1(x)=l_1$ for some 
$x\in\gamma_1\cap\Omega_*^u$}\}.$$ 
Given $\mathcal Q(l_1,\ldots,l_{i})$,
for
each $\gamma_{i}\in\mathcal Q(l_1,\ldots,l_{i})$ set
$$\mathcal Q(\gamma_{i},l_{i+1})= \{\gamma_{i+1}\subset\gamma_{i}\colon
f^{l_1+\cdots+l_{i}+l_{i+1}}\gamma_{i+1}\in\tilde\Gamma^u\text{ and $k_{i+1}(x)=l_{i+1}$ for some 
$x\in\gamma_{i+1}\cap\Omega_*^u$}\},$$
and define 
$$\mathcal Q(l_1,\ldots,l_{i+1})=
\bigcup_{\gamma_{i}\in \mathcal Q(l_1,\ldots,l_{i}) }
\mathcal Q(\gamma_{i},l_{i+1}).$$
Obviously,
\begin{equation}\label{cardinal}
\mathcal Q(\gamma_{i},l_{i+1})<2^{l_{i+1}}.
\end{equation}

If $x\in\Omega_*^u$, then 
for each $i=1,2,\ldots$ there exists a unique element of $\mathcal Q(k_1(x),k_2(x),\ldots,k_i(x))$
containing $x$.
Hence  
$$\Omega_*^u\subset
\bigcup_{L=\eta^{i-1}}^\infty\bigcup_{l_1+\cdots+l_i=L}\bigcup_{\gamma_i\in \mathcal Q(l_1,\ldots,l_i)}\gamma_{i}.$$

Now, let $1\leq p<q$ and define
$$\Omega_*^{(p)}=\{f^{k_1(x)+k_2(x)+\cdots+k_p(x)}x\colon x\in\Omega_*^u\}.$$
If $x\in\Omega_*^{(p)}$, then
$k_{p+1}(x)+k_{p+2}(x)+\cdots+k_{q}(x)\geq
\eta^{q-1}.$
Hence
$$\Omega_*^{(p)}\subset
\bigcup_{\gamma_{q}\in \mathcal Q(l_1,\ldots, l_q)}f^{l_1+\cdots+l_p}\gamma_{q}=
\bigcup_{L=\eta^{q-1}}^\infty\bigcup_{l_{p+1}+\cdots+l_q=L}\bigcup_{\gamma_{q}\in \mathcal Q(l_1,\ldots,l_q)}f^{l_1+\cdots+l_p}\gamma_{q}.$$
From the countable stability and the $f$-invariance
of $\dim_H^u$,
$\dim_H^u(\Omega_*^u)=\dim_H^u(\Omega_*^{(p)}).$
To get a better estimate, we shall work with large $p$.

Let $d\in(2/\log (1/b),1)$. For each $i\geq p$ we have
\begin{align*}
\sum_{\gamma_{i+1}\in\mathcal Q(l_1,\ldots,l_{i+1})}
{\rm length}(f^{l_1+\cdots+l_p}\gamma_{i+1})^d&=\sum_{\gamma_{i}\in
\mathcal Q(l_1,\ldots,l_{i})}
{\rm length}(f^{l_1+\cdots+l_p}\gamma_{i})^d\\
&\times\sum_{\gamma_{i+1}\in\mathcal
Q(\gamma_{i},l_{i+1}) } \frac{{\rm length}(f^{l_1+\cdots+l_p}\gamma_{i+1})^d}
{{\rm length}(f^{l_1+\cdots+l_p}\gamma_{i})^d}.\end{align*}
On the second sum of the fractions, 
we have $f^{k_1+\cdots+k_{i}}\gamma_{i}\in\tilde\Gamma^u$ and 
${\rm length}(f^{k_1+\cdots+k_{i}}\gamma_{i+1})< 2b^{\frac{k_{i+1}}{9}}.$
From this and the bounded distortion in Lemma \ref{global},
\begin{equation}\label{d4}\frac{{\rm length}(f^{l_1+\cdots+l_p}\gamma_{i+1})}{{\rm length}(f^{l_1+\cdots+l_p}\gamma_i)}\leq
C\cdot\frac{{\rm length}(f^{l_1+\cdots+l_{i}}\gamma_{i+1})}{{\rm length}(f^{l_1+\cdots+l_{i}}\gamma_i)}\leq
3b^{\frac{l_{i+1}}{9}}.\end{equation} 
Using  \eqref{cardinal} \eqref{d4} and $d\in(2/\log (1/b),1)$,
$$\sum_{\gamma_{i+1}\in\mathcal Q(\gamma_{i},l_{i+1})
}\frac{{\rm length}(f^{l_1+\cdots+l_p}\gamma_{i+1})^d}{{\rm length}(f^{l_1+\cdots+l_p}\gamma_{i})^d} \leq
\#\mathcal Q(\gamma_{i},l_{i+1}) 3^db^{\frac{dl_{i+1}}{10}}<b^{\frac{dl_{i+1}}{20}}.$$ 
Plugging this
into the right-hand-side of the above equality we get
\begin{equation}\label{d-1}
\sum_{\gamma_i\in\mathcal Q(l_1,\ldots,l_{i+1})}{\rm length}(f^{l_1+\cdots+l_p}\gamma_{i+1})^d
\leq b^{\frac{d l_{i+1}}{20}}
\sum_{\gamma_{i}\in\mathcal Q(l_1,\ldots,l_{i})}{\rm length}(f^{l_1+\cdots+l_p}\gamma_{i})^d
.\end{equation} 
Using \eqref{d-1} inductively 
yields
\begin{equation*}
\sum_{\gamma_q\in\mathcal Q(l_1,\ldots,l_q)}{\rm length}(f^{l_1+\cdots+l_p}\gamma_q)^d
\leq L_pb^{\frac{d}{20}(l_{p+1}+\cdots+l_q)},\end{equation*} 
where $$L_p=\sum_{\gamma_{p+1}\in\mathcal Q(\gamma_p,l_{p+1})
}\frac{{\rm length}(f^{l_1+\cdots+l_p}\gamma_{p+1})^d}{{\rm length}(f^{l_1+\cdots+l_p}\gamma_{p})^d}.$$
Hence
\begin{align*}
\sum_{L=\eta^{q-1}}^\infty\sum_{l_{p+1}+\cdots+l_q=L}\sum_{\gamma_q\in\mathcal
Q(l_1,\ldots,l_q)}{\rm length}(f^{l_1+\cdots+l_p}\gamma_q)^d&\leq 
L_p\sum_{L=\eta^{q-1}}^\infty
b^{\frac{dL}{20}}\#\left\{
(l_{p+1},\ldots,l_{q})\colon \sum_{i=p+1}^q l_i=L\right\}.
\end{align*}
To estimate the right-hand side we use the following from Stirling's formula for factorials: for sufficiently small $\chi>0$ there exist $c(\chi)>0$ with $c(\chi)\to0$ as $\chi\to0$ such that for any two positive integers $p$, $q$ with $q/p\leq\chi$ one has
$\left(\begin{smallmatrix}p+q\\q\end{smallmatrix}\right)\leq e^{c(\chi)p}$.

The number of all
feasible $(l_{p+1},\ldots,l_{q})$ with $\sum_{i=p+1}^{q} l_{i}=L$ is bounded by the
number of ways of dividing $L$ objects into $q-p$ groups, which is
$\left(\begin{smallmatrix}L+q-p\\q-p\end{smallmatrix}\right)$.
Since $l_i\geq\eta^p$ for $i=p+1,\ldots,q$, we have $(q-p)/L\leq \eta^{-p}$,
which goes to $0$ as $p\to\infty$.
In particular,
there exists $p_0$ such that for all $p$, $q$ with $p_0\leq p<q$,
$$\#\left\{
(l_{p+1},\ldots,l_{q})\colon \sum_{i=p+1}^q l_{i}=L\right\}\leq \begin{pmatrix}L+q-p\\q-p\end{pmatrix}
\leq b^{-\frac{dL}{30}}.$$
\begin{align*}
\sum_{L=\eta^{q-1}}^\infty\sum_{l_{p+1}+\cdots+l_{q}=L}\sum_{\omega_q\in\mathcal
Q(l_1,\ldots,l_q)}{\rm length}(f^{l_1+\cdots+l_p}\gamma_q)^d
&\leq
L_p\sum_{L=\eta^{q-1}}^\infty      b^{\frac{bdL}{60}}.  
\end{align*} 
 The summand of the right-hand-side decays exponentially
in $q$, and so the
Hausdorff $d$-measure of $\Omega_*^{(p)}$ is zero.\end{proof}

From the lower estimate in Sect.\ref{lowest}, Proposition \ref{c} and Proposition \ref{new} we obtain
$$\lim_{\varepsilon\to0}d^u_\varepsilon\leq B_\varphi^u(\beta)=\max\{\dim_H^u(\Pi_1),\dim_H^u(\Pi_2)\}\leq\max\{\lim_{\varepsilon\to0}d^u_\varepsilon,2/\log(1/b)\}.$$
Since $I_\varphi\setminus I_\varphi'$ we have $B_\varphi^u(\beta)>2/\log(1/b)$, and so the above two inequalities are equalities. 
This completes the proof of the theorem. $\qed$

\subsection*{Acknowledgments}
Partially supported by the Grant-in-Aid for Young Scientists (B) of the JSPS, Grant No.23740121.

\bibliographystyle{amsplain}

\end{document}